\documentclass[12 pt]{article}
\usepackage{amsmath}
\usepackage{amsthm}
\usepackage{float}
\usepackage{graphicx}
\usepackage{hyperref}
\usepackage{amssymb}
\usepackage{color}
\usepackage{comment}

\newcommand{\floor}[1]{\lfloor {#1} \rfloor}

\newtheorem{lem}{Lemma}

\newtheorem{thm}{Theorem}

\title{Sums and differences of sets: a further improvement over AlphaEvolve}
\author{Fan Zheng}

\begin{document}
\maketitle

\begin{abstract}
We present a new advancement in the sum and difference of sets problem, which improves upon recent results by both DeepMind's AlphaEvolve ($\theta = 1.1584$) and subsequent explicit constructions ($\theta = 1.173050$). In this work, we construct a sequence of $U$ sets which in the limit establishes a new lower bound of $\theta = 1.173077$.
\end{abstract}


\medskip


\section{Introduction}

Let \( K > 1 \) be a real number. In \cite{gyarmati2007}, the authors studied the growth of sumsets and difference sets of integers and posed the problem of determining the largest possible exponent \( \theta \) such that, for arbitrarily large finite sets \( A \) and \( B \) of integers, the following holds for some constant \( c(K) > 0 \):
\begin{equation}
    |A + B| \leq K|A| \quad \text{and} \quad |A - B| \geq c(K) \cdot |A + B|^\theta.
    \label{eq:theta}
\end{equation}

They showed that for any finite set \( U \subset \mathbb{N} \) containing zero, the inequality
\[
\theta \geq 1 + \frac{\log\left(\frac{|U - U|}{|U + U|}\right)}{\log(2\max(U) + 1)}
\]
provides a lower bound on the attainable exponent \( \theta \). Their constructions were based on a combinatorial set \( V = V(m, L) \subset \mathbb{N}^m \) defined by
\[
V(m, L) = \left\{ (x_1, \ldots, x_m) \in \mathbb{N}^m : x_1 + \cdots + x_m \leq L \right\},
\]
which satisfies \( |V(m, L)| = \binom{m + L}{m} \). An injective mapping \( f \) was used to transform vectors into integers via
\[
f(x_1, \ldots, x_m) = \sum_{k=0}^{m-1} x_k \cdot L_k, \quad \text{with } L_0 = 1, \quad L_k = 2L \cdot L_{k-1} + 1 \text{ for } k > 0,
\]
ensuring that \( f \) is injective on both \( V + V \) and \( V - V \). The image \( U = f(V) \) then served as the key set for evaluating \( \theta \). This approach established the bound \( \theta \geq 1.14465 \).

More recently, DeepMind's AlphaEvolve system \cite{deepmind} applied large language models to this problem and achieved an improved bound of \( \theta = 1.1584 \) using a set of 54{,}265 integers. A subsequent explicit construction \cite{gerbicz2025} raised this to \( \theta = 1.173050 \) using a \( U \) set with over $10^{43{,}546}$ elements.

In this work, we construct a sequence of \( U \) sets which, in the limit, establishes a new lower bound of
\[
\theta = 1.173077.
\]
Our approach builds on the previous combinatorial framework and extends it to \( U \) sets of arbitrarily large size. To achieve this, we analyze the asymptotic behavior of sets of the form
\[
W(m, L, B) = \left\{ x = (x_1, \ldots, x_m) \in V(m, L): x_1 \leq B, \ldots, x_m \leq B \right\},
\]
where each coordinate is individually bounded. Estimating the size of such sets for large \( m \) and \( L \) requires a new ingredient: large deviation estimates, which allow us to quantify the distribution of constrained integer partitions and obtain precise asymptotic counts necessary for bounding \( \theta \).

\section{Large deviation estimates}
Let $m$, $L$ and $B$ be nonnegative integers. Define:
\[
W(m, L, B) = \left\{ x = (x_1, \ldots, x_m) \in V(m, L): x_1 \leq B, \ldots, x_m \leq B \right\}.
\]
\begin{lem}\label{large-deviation}
\[
\lim_{m\to\infty} \frac{\log|W(m, \floor{rm}, B)|}{m} = \log(B + 1) - I(r, B),
\]
where
\[
I(c, B) =
\begin{cases}
0, & c \ge B/2,\\
\sup_t \left( tc - \log \frac{1 + e^t + \dots + e^{Bt}}{B + 1} \right), & c < B/2.
\end{cases}
\]
\end{lem}
\begin{proof}
We regard each $x_i$ as an independent random variable uniformly distributed over the set $\{0, 1, \dots, B\}$. Then the mean value of each $x_i$ is $B/2$. By the law of large numbers, as $m \to \infty$,
\[
\frac{|W(m, \floor{mB/2}, B)|}{(B + 1)^m} = \mathbb P(x_1 + \dots + x_m \le mB/2) \to \frac12.
\]

If $r \ge B/2$, then $W(m, \floor{rm}, B) \supseteq W(m, \floor{mB/2}, B)$, so
\[
\liminf_{m\to\infty} \frac{|W(m, \floor{rm}, B)|}{(B + 1)^m} \ge \frac12.
\]
Then
\[
\liminf_{m\to\infty} (\log |W(m, \floor{rm}, B)| - m\log(B + 1)) \ge -\log2.
\]
Then
\[
\liminf_{m\to\infty} \frac{\log |W(m, \floor{rm}, B)|}m \ge \log(B + 1).
\]
On the other hand,
\[
\frac{\log |W(m, \floor{rm}, B)|}m \le \frac{\log (B + 1)^m}m = \log(B + 1),
\]
so
\[
\lim_{m\to\infty} \frac{\log|W(m, \floor{rm}, B)|}{m} = \log(B + 1).
\]

If $r < B/2$, note that $c \mapsto I(c, B)$ is the Legendre transformation of the generating function of the random variable $x_i$, so by Cramer's theorem \cite{Cramer},
\[
\lim_{m\to\infty} \frac{\log\mathbb P(x_1 + \dots + x_m \le mB/2)}{m} = -I(r, B).
\]
On the other hand,
\[
\log \frac{\mathbb P(x_1 + \dots + x_m \le mB/2)}m = \frac{\log|W(m, \floor{mB/2}, B)|}m - \log(B + 1),
\]
so
\[
\lim_{m\to\infty} \frac{\log|W(m, \floor{rm}, B)|}{m} = \log(B + 1) - I(r, B).
\]
\end{proof}

\begin{lem}\label{semicont}
For any $c_0 > 0$ we have
\[
\limsup_{c\to c_0} I(c, B) \le I(c_0, B).
\]
\end{lem}
\begin{proof}
If $c_0 > B/2$ then both sides are zero so the inequality trivially holds.

If $c_0 < B/2$, by the AM--GM inequality, for $t \ge 0$ we have
\[
tc_0 - \log \frac{1 + e^t + \dots + e^{Bt}}{B + 1}
\le tc_0 - \log e^{\frac{Bt}2} = t\left( c_0 - \frac B2 \right) \le 0,
\]
with equality attained when $t = 0$. Hence for any $c < B/2$ we have
\[
I(c, B) = \sup_{t\le0} \left( tc - \log \frac{1 + e^t + \dots + e^{Bt}}{B + 1} \right).
\]
On the other hand, if $c > c_0/2$ and $t < -\frac{2\log(B + 1)}{c_0}$ then
\[
tc - \log \frac{1 + e^t + \dots + e^{Bt}}{B + 1}
< -\log(B + 1) - \log\frac1{B + 1} = 0,
\]
so for $c \in (c_0/2, B/2)$ we have
\[
I(c, B) = \sup_{t\in\left[-\frac{2\log(B + 1)}{c_0},0\right]} \left( tc - \log \frac{1 + e^t + \dots + e^{Bt}}{B + 1} \right).
\]
Then for any $\epsilon > 0$ there is $t_0\in\left[-\frac{2\log(B + 1)}{c_0},0\right]$ such that
\[
t_0c - \log \frac{1 + e^{t_0} + \dots + e^{Bt_0}}{B + 1} > I(c, B) - \epsilon.
\]
Then
\[
t_0c_0 - \log \frac{1 + e^{t_0} + \dots + e^{Bt_0}}{B + 1} > I(c, B) - \epsilon - \frac{2\log(B + 1)}{c_0}|c - c_0|.
\]
Then there is $\delta > 0$ such that for any $c \in (c_0 - \delta, c_0 + \delta)$ we have
\[
I(c_0, B) \ge t_0c_0 - \log \frac{1 + e^{t_0} + \dots + e^{Bt_0}}{B + 1} > I(c, B) - 2\epsilon,
\]
showing the claim.

If $c_0 = B/2$, then for any $c \ge c_0$, $I(c, B) = 0$, so
\[
\limsup_{c\to c_0+} I(c, B) = I(c_0, B) = 0.
\]
For $c \in (c_0/2, c_0)$ we also have
\[
t_0c_0 - \log \frac{1 + e^{t_0} + \dots + e^{Bt_0}}{B + 1} > I(c, B) - \epsilon - \frac{2\log(B + 1)}{c_0}|c - c_0|.
\]
On the other hand,
\[
t_0c_0 - \log \frac{1 + e^{t_0} + \dots + e^{Bt_0}}{B + 1} \le t_0c_0 - \log e^{\frac{Bt_0}2} = 0,
\]
so
\[
I(c, B) \le \epsilon + \frac{2\log(B + 1)}{c_0}|c - c_0|,
\]
and then
\[
\limsup_{c\to c_0-} I(c, B) \le 0 = I(c_0, B).
\]
\end{proof}

\section{Main result}
\begin{thm}
For the sums and differences of sets problem, we have $\theta \ge 1.173077$.
\end{thm}
\begin{proof}
Following \cite{gerbicz2025}, we construct the set $U$ as the image of $W(m, L, B)$ under the map $g: (x_1, \dots, x_m) \mapsto \sum_{k=0}^{m-1} x_k(2B + 1)^k$,
with $B > 0$ fixed and $m$ and $L\to\infty$. Then according to \cite{gyarmati2007}, we have
\[
\theta \ge 1 + \frac{\log d(U) - \log s(U)}{\log q(U)},
\]
where $d(U) = |U - U|$, $s(U) = |U + U|$ and $q(U) = 2\max(U) + 1$.

Since each coordinate $x_k \le B$, we have
\[
\max U \le B\sum_{k=0}^{m-1} (2B + 1)^k = B\frac{(2B + 1)^m - 1}{2B},
\]
so $q(U) = 2\max(U) + 1 \le (2B + 1)^m$, and then
\begin{equation}\label{q}
\frac{\log q(U)}m \le \log(2B + 1).
\end{equation}

In \cite{gerbicz2025} it was shown that $s(U) = W(m, 2L, 2B)$. We now let $L=\floor{rm}$.
Then $2L \le \floor{2rm}$, so $W(m, 2L, 2B) \subseteq W(m, \floor{2rm}, 2B)$,
so by Lemma \ref{large-deviation},
\begin{equation}\label{s}
\limsup_{m\to\infty} \frac{\log s(U)}m
\le \lim_{m\to\infty} \frac{W(m, \floor{2rm}, 2B)}m = \log(2B + 1) - I(2r, 2B).
\end{equation}

In \cite{gerbicz2025} it was also shown that
\[
d(U) = \sum_{k=0}^{\min(m,L)} {m \choose k}|W(k, L - k, B - 1)||W(m - k, L, B)|.
\]
Then for any $k = 0, \dots, \min(m, L)$ we have
\[
d(U) \ge {m \choose k}|W(k, L - k, B - 1)||W(m - k, L, B)|.
\]
We now let $k = \floor{aL}$, where $a \in (0, \min(1, 1/r))$ is to be determined.
Then
\[
L - k \ge (1 - a)L \ge \frac{1 - a}ak\ge \floor{\frac{1 - a}ak},
\]
so $W(k, L - k, B - 1) \supseteq W(k, \floor{\frac{1 - a}ak}, B - 1)$.
By Lemma \ref{large-deviation} then,
\[
\liminf_{m\to\infty} \frac{\log|W(k, L - k, B - 1)|}k
\ge \log B - I\left( \frac{1-a}a, B - 1 \right).
\]
Since
\[
k = \floor{aL} > aL - 1 = a\floor{rm} - 1 > arm - a - 1,
\]
we have $\frac km \to ar$ as $m \to \infty$, so
\begin{equation}\label{W1}
\liminf_{m\to\infty} \frac{\log|W(k, L - k, B - 1)|}m
\ge ar\left( \log B - I\left( \frac{1-a}a, B - 1 \right) \right).
\end{equation}

Note that in the set $W(m, k, 1)$, exactly $k$ of the $m$ coordinates are 1,
and the rest are 0, so ${m \choose k} = W(m, k, 1)$. Since $\frac km \to ar$
as $m \to \infty$, for any $c < ar$, if $m$ is sufficiently large then $k > cm$,
so $W(m, k, 1) \supseteq W(m, cm, 1)$. By Lemma \ref{large-deviation} then,
\begin{equation}\label{C}
\liminf_{m\to\infty} \frac{\log{m \choose k}}m
\ge \log2 - I\left( c, 1 \right).
\end{equation}
By Lemma \ref{semicont}, we then have
\[
\liminf_{m\to\infty} \frac{\log{m \choose k}}m
\ge \log2 - \limsup_{c\to ar-} I\left( c, 1 \right) \ge \log2 - I(ar, 1).
\]

Likewise,
\[
m - k < m - aL + 1 < m - arm + a + 1
\]
so
\[
\frac{L}{m - k} > \frac{rm - 1}{m - arm + a + 1} \to \frac{r}{1 - ar}
\]
as $m \to \infty$, so
\[
\liminf_{m\to\infty} \frac{\log|W(m - k, L, B)|}{m - k} \ge \log(B + 1) - I\left(\frac r{1 - ar}, B\right).
\]
Since $\frac{m-k}m \to 1 - ar$ as $m \to \infty$,
\begin{equation}\label{W2}
\liminf_{m\to\infty} \frac{\log|W(m - k, L, B)|}m \ge (1 - ar)\left( \log(B + 1) - I\left(\frac r{1 - ar}, B\right) \right).
\end{equation}
Adding \eqref{C}, \eqref{W1} and \eqref{W2} we get
\begin{align}
\nonumber
\liminf_{m\to\infty} \frac{d(U)}m
&\ge \log2 + ar\log B + (1 - ar)\log(B + 1) - I(ar, 1)\\
&- arI\left( \frac{1-a}a, B - 1 \right) - (1 - ar)I\left(\frac r{1 - ar}, B\right).
\label{d}
\end{align}

Putting \eqref{d}, \eqref{s} and \eqref{q} together we have
\[
\theta \ge \tfrac{\log2 + ar\log B + (1 - ar)\log(B + 1) - I(ar, 1) - arI\left( \frac{1-a}a, B - 1 \right) - (1 - ar)I\left(\frac r{1 - ar}, B\right) + I(2r, 2B)}{\log(2B + 1)}.
\]

We first maximize in $a$ over $(0, \min(1, 1/r))$, then maximize in $r > 0$,
and finally maximize in $B$. Thus theoretically we would have $\theta \ge \theta_0$,
where
\begin{align*}
\theta_0 &= \sup_{B\ge1}\sup_{r>0}\sup_{a\in(0,\min(1,1/r))}\\
&\tfrac{\log2 + ar\log B + (1 - ar)\log(B + 1) - I(ar, 1) - arI\left( \frac{1-a}a, B - 1 \right) - (1 - ar)I\left(\frac r{1 - ar}, B\right) + I(2r, 2B)}{\log(2B + 1)}.
\end{align*}
In reality we maximize $\theta - 1$ and take the supreme in a more restricted yetrelevant region:
\begin{align*}
\theta_0 - 1 &\ge \sup_{B=3}^{10}\sup_{r\in(0.5,2)}\sup_{a\in(0,\min(1,1/r))}\\
&\tfrac{\log2 + ar\log B + (1 - ar)\log(B + 1) - I(ar, 1) - arI\left( \frac{1-a}a, B - 1 \right) - (1 - ar)I\left(\frac r{1 - ar}, B\right) - \log(2B + 1) + I(2r, 2B)}{\log(2B + 1)}.
\end{align*}

We compute this (restricted) supreme in MATLAB\textsuperscript{\textregistered} 2024b
using the code listed in the next section, with the result given below:
\begin{table}[H]
\centering
\begin{tabular}{|c|c|c|c|c|}
\hline
$B$ & $\epsilon=10^{-4}$ & $\epsilon=10^{-6}$ & $\epsilon=10^{-8}$ & $\epsilon=10^{-10}$ \\
\hline
3 & 0.168700179529982 & 0.168700179627153 & 0.168700179627163 & 0.168700179627163 \\
\hline
4 & 0.172137890281832 & 0.172137890014041 & 0.172137890014121 & 0.172137890014121 \\
\hline
5 & 0.173077285664668 & 0.173077279785044 & {\color{blue}0.173077279785136} & {\color{blue}0.173077279785136} \\
\hline
6 & 0.172855934276473 & 0.172855932676982 & 0.172855932676998 & 0.172855932676998 \\
\hline
7 & 0.172060244111987 & 0.172060243359401 & 0.172060243360376 & 0.172060243360376 \\
\hline
8 & 0.170975347606027 & 0.170975345188620 & 0.170975345189401 & 0.170975345189401 \\
\hline
9 & 0.169749937706638 & 0.169749936705626 & 0.169749936705622 & 0.169749936705623 \\
\hline
10& 0.168465310171444 & 0.168465310634718 & 0.168465310634737 & 0.168465310634737 \\
\hline
\end{tabular}
\caption{Numerical data table}
\end{table}
Here $\epsilon$ is the tolerance used by MATLAB\textsuperscript{\textregistered}'s \texttt{fminunc} and \texttt{fminbnd} functions. Note that for $B = 5$ and $\epsilon = 10^{-8}$ or $10^{-10}$ the maximal values attained, marked in blue, are identical, up to 15 decimal places. Hence we have
\[
\theta \ge \theta_0 = 1.173077\dots
\]
\end{proof}

\section{MATLAB\textsuperscript{\textregistered} Code}
The main script is as follows:
\begin{verbatim}
B = 1:10;
a = zeros(max(B), 4);
for x = B
    a(x,:) = theta(x);
end
a
\end{verbatim}
The function \texttt{theta(B)} computes the optimal value for a fixed $B$,
with the tolerance set at $10^{-4}$, $10^{-6}$, $10^{-8}$ and $10^{-4}$ successively:
\begin{verbatim}
function ret = theta(B)
    ret = zeros(1, 4);
    ret(1) = theta_eps(B, 1e-4);
    ret(2) = theta_eps(B, 1e-6);
    ret(3) = theta_eps(B, 1e-8);
    ret(4) = theta_eps(B, 1e-10);
end
\end{verbatim}
The function \texttt{theta(B, eps)} computes the optimal value for a fixed $B$
at a fixed tolerance $\epsilon$:
\begin{verbatim}
function ret = theta_eps(B, eps)
    fun = @(r) -(logD(r, B, eps) - logW(2*r, 2*B, eps));
    options = optimset('TolX', eps);
    [~, val] = fminbnd(fun, 0.5, 2, options);
    ret = -val/log(2*B+1);
end
\end{verbatim}
The function \texttt{logD(r, B, eps)} maximizes the value
\[
\log2 + ar\log B + (1 - ar)\log(B + 1) - I(ar, 1) - arI\left( \frac{1-a}a, B - 1 \right) - (1 - ar)I\left(\frac r{1 - ar}, B\right)
\]
in $a$ over the interval $(0, \min(1, 1/r))$, at a fixed tolerance $\epsilon$:
\begin{verbatim}
function ret = logD(r, B, eps)
    fun = @(a) -term(r, B, a, eps);
    options = optimset('TolX', eps);
    [~, val] = fminbnd(fun, 0, min(1, 1/r), options);
    ret = -val;
end
\end{verbatim}
The function \texttt{term(r, B, a, eps)} computes the value above, at a fixed tolerance $\epsilon$:
\begin{verbatim}
function ret = term(r, B, a, eps)
    ret = logW(a*r, 1, eps);
    ret = ret + a*r * logW(1/a-1, B-1, eps);
    ret = ret + (1-a*r) * logW(1/(1/r-a), B, eps);
end
\end{verbatim}
The function \texttt{logW(c, B, eps)} computes the value $\log(B + 1) - I(c, B)$,
at a fixed tolerance $\epsilon$:
\begin{verbatim}
function ret = I(c, B, eps)
    if c < B/2
        ret = log(B + 1) - sup(c, B, eps);
    else
        ret = log(B + 1);
    end
end
\end{verbatim}
Finally the function \texttt{sup(c, B, eps)} computes the supremum
\[
\sup_t \left( tc - \log \frac{1 + e^t + \dots + e^{Bt}}{B + 1} \right)
\]
for fixed $c$ and $B$, at a fixed tolerance $\epsilon$:
\begin{verbatim}
function ret = sup(c, B, eps)
    logE = @(t) log(mean(exp(t * (0:B))));
    opt = optimoptions('fminunc', 'Display','notify',...
        'OptimalityTolerance',eps,...
        'StepTolerance',eps,...
        'FunctionTolerance',eps);
    [~, ret] = fminunc((@(t) logE(t) - t * c), 0, opt);
    ret = -ret;
end\end{verbatim}


\begin{thebibliography}{9}
\bibitem{gerbicz2025}
Robert Gerbicz,
\textit{Sums and differences of sets (improvement over AlphaEvolve)},
arXiv preprint arXiv:2505.16105, 2025.\\
\url{https://arxiv.org/abs/2505.16105}

\bibitem{gyarmati2007}
K. Gyarmati, F. Hennecart, and I. Z. Ruzsa,
\textit{Sums and differences of finite sets},
Functiones et Approximatio Commentarii Mathematici, \textbf{37}(1):175–186, 2007.\\
\url{https://gyarmatikati.web.elte.hu/publ/sumdiffv.pdf}

\bibitem{klenke2008}
Achim Klenke,
\textit{Probability Theory},
Springer, Berlin, 2008, pp. 508.\\
ISBN: 978-1-84800-047-6.\\
DOI: \href{https://doi.org/10.1007/978-1-84800-048-3}{10.1007/978-1-84800-048-3}

\bibitem{deepmind}
Alexander Novikov, Ng\^an Vu, Marvin Eisenberger, Emilien Dupont, Po-Sen Huang, Adam Zsolt Wagner, Sergey Shirobokov, Borislav Kozlovskii, Francisco J. R. Ruiz, Abbas Mehrabian, M. Pawan Kumar, Abigail See, Swarat Chaudhuri, George Holland, Alex Davies, Sebastian Nowozin, Pushmeet Kohli, and Matej Balog,\\
\textit{AlphaEvolve: A coding agent for scientific and algorithmic discovery}, 2025.\\
\url{http://bit.ly/4jkeOoX}\end{thebibliography}
\end{document}